\PassOptionsToPackage{fleqn}{amsmath}
\documentclass[12pt]{amsart}

\usepackage[margin=1in, a4paper]{geometry}
\usepackage[english]{babel}
\usepackage[fleqn]{amsmath}
\usepackage[foot]{amsaddr}
\usepackage[scr=dutchcal]{mathalpha}
\usepackage[perpage]{footmisc}

\usepackage{
	amsthm,
	amssymb,
	bbm,
	hyperref,
	cleveref,
	enumitem,
    tikz-cd
}

\newtheorem{theorem}{Theorem}[section]
\newtheorem*{thm*}{Theorem}

\newtheorem{lemma}[theorem]{Lemma}
\newtheorem{proposition}[theorem]{Proposition}

\theoremstyle{definition}
\newtheorem{definition}[theorem]{Definition}
\AtBeginEnvironment{definition}{%
	\pushQED{\qed}%
}
\AtEndEnvironment{definition}{\popQED\enddefinition}

\newtheorem{remark}[theorem]{Remark}
\AtBeginEnvironment{remark}{%
	\pushQED{\qed}%
}
\AtEndEnvironment{remark}{\popQED\endremark}

\newtheorem{example}[theorem]{Example}
\AtBeginEnvironment{example}{%
	\pushQED{\qed}%
}
\AtEndEnvironment{example}{\popQED\endexample}

\renewcommand{\geq}{\geqslant}
\renewcommand{\leq}{\leqslant}
\newcommand{\sub}{\subseteq}
\newcommand{\eps}{\varepsilon}
\renewcommand{\phi}{\varphi}

\newcommand{\mc}[1]{{\mathcal{#1}}}
\newcommand{\ms}[1]{{\mathscr{#1}}}

\newcommand{\B}{\mathscr{B}}
\newcommand{\C}{\mathbb{C}}
\newcommand{\E}{\mathbb{E}}
\newcommand{\F}{\mathcal{F}}
\renewcommand{\H}{\mathscr{H}}
\newcommand{\K}{\mathbb{K}}
\newcommand{\N}{\mathbb{N}}
\renewcommand{\P}{\mathbb{P}}
\newcommand{\R}{\mathbb{R}}

\let\Re\relax
\DeclareMathOperator{\Re}{\operatorname{Re}}
\renewcommand{\Im}{\operatorname{Im}}
\newcommand{\one}{\mathbbm{1}}
\newcommand{\trace}{\operatorname{trace}}
\newcommand{\cov}{\operatorname{cov}}
\newcommand{\Var}{\operatorname{Var}}
\newcommand{\PVar}{\operatorname{PVar}}
\newcommand{\pcov}{\operatorname{pcov}}

\renewcommand{\d}{\operatorname{d}\hspace{-0.5mm}}
\newcommand{\norm}[1]{\left\lVert#1\right\rVert}
\newcommand{\opnorm}[1]{\norm{#1}_{\text{op}}}
\newcommand{\ptnorm}[1]{\norm{#1}_{\text{PT}}}
\newcommand{\lownorm}[1]{\lVert#1\rVert}
\newcommand{\brak}[1]{\left\langle#1\right\rangle}

\newcommand{\olsi}[1]{\,\overline{\!{#1}}}
\newcommand{\oolsi}[1]{\mskip.5\thinmuskip\overline{\mskip-.5\thinmuskip{#1}\mskip-.1\thinmuskip}\mskip.5\thinmuskip}
\renewcommand{\bar}[1]{\oolsi{#1}}

\newcommand{\br}{{B_\R}}
\newcommand{\hr}{{H_\R}}
\newcommand{\brp}{{B_\R'}}
\newcommand{\test}{{C_0^\infty(D)}}
\newcommand{\rd}{{\R^d}}
\newcommand{\fls}{(-\Delta)^s}
\newcommand{\lso}{{\mathscr L_s(D)}}
\newcommand{\ls}{{\mathscr L_s}}
\newcommand{\lsd}{{\mathscr L_{-s}(D)}}
\newcommand{\lsm}{{\mathscr L_{-s}}}
\newcommand{\lsp}{{\mathscr L_{-s}'}}
\newcommand{\ltp}{{\ms L_{-t}'}}
\newcommand{\ltd}{{L^2(D)}}
\newcommand{\lt}{{L^2}}
\newcommand{\ds}{{(-\Delta)^{-s}}}

\title{Complex abstract Wiener spaces}

\author{Tess J. van Leeuwen}
\email{t.j.vanleeuwen@uu.nl}

\author{Wioletta M. Ruszel}
\address{Utrecht University, Budapestlaan 6, 3584 CD Utrecht, The Netherlands}
\email{w.m.ruszel@uu.nl}


\begin{document}

\begin{abstract}
Real abstract Wiener spaces (AWS)  were originally defined by Gross using measurable norms, as a generalisation of the theory of advanced integral calculus in infinite dimensions as introduced by Cameron and Martin. In this paper we present a rigorous, complete and self-contained general framework for $\mathbb{K}$-AWS, where $\mathbb{K} \in \{\mathbb{R},\mathbb{C}\}$ using the language of characteristic functions instead of measurable norms. In particular, we will prove that
$X$ is a symmetric $H$-valued Gaussian field over $\mathbb{K}$ iff the covariance function can be written in terms of some non-negative, self-adjoint trace class operator, and that the existence and uniqueness of $X$ is equivalent to the  $\mathbb{K}$-AWS.
Finally we will relate the $\mathbb{C}$-AWS to the $\mathbb{R}$-AWS by way of a real structure, which is a real linear, complex anti-linear involution on a complex vector space. This allows for 
a commutative relation between the real and complex Gaussian fields and the real and complex abstract Wiener spaces. We will construct specific examples which fall under this framework like the complex Brownian motion, complex Feynman-Kac formula and complex fractional Gaussian fields.
\end{abstract}
\maketitle

\section{Introduction}

Abstract Wiener spaces (AWS) provide a natural generalisation of classical probability spaces for describing Gaussian measures. An AWS is a triple $(H, B, \nu)$ consisting of a separable Hilbert space $H$, a separable Banach space $B$ (densely containing $H$), and a Gaussian probability measure $\nu$ on $B$ whose covariance structure is that of a standard Gaussian random element on $H$. The challenge in defining the latter for an infinite-dimensional Hilbert space arises from the fact that, if expressed as 
\[
Z = v_1 Z_1 + v_2 Z_2 + \ldots + v_n Z_n + \ldots,
\]
where $(Z_n)_{n\in\N}$ is a sequence of i.i.d. Gaussian random variables and $(v_n)_{n\in\N}$ is an orthonormal basis for $H$, the resulting random element $Z$ has almost surely infinite $H$-norm. As an alternative we thus consider a Gaussian random element defined not on the Hilbert space $H$, but on a larger Banach space $B$, such that its covariance structure is determined by $H$.

Real abstract Wiener spaces were originally defined by Gross in \cite{gross} as a generalisation of the theory of advanced integral calculus in infinite dimensions as introduced by Cameron and Martin in \cite{CM1, CM2}. They considered the probability measure $\nu$ generating the Brownian motion, defined on the underlying Banach space $B$ of continuous functions $x: [0, T] \to \R$, starting in zero. Only from work by Segal \cite{segal1, segal2} it became clear that the essential element for the formulas developed by Cameron and Martin was the Euclidean (inner product) structure of the Hilbert space $H$ of $(1, 2)$-Sobolev functions $x: [0, T] \to \R$ starting in the origin. Gross then generalised this concept of a triple $(H, B, \nu)$ for any separable real Hilbert space $H$. 

In an earlier article \cite{gross_meas}, Gross introduced measurable norms. In \cite{gross} the key insight was that if $B$ is the Banach space completion of a Hilbert space $H$ endowed with a measurable norm (necessarily weaker than the $H$-norm), then the Gaussian distribution on $B$ becomes countably additive. Properties of Gaussian measures on Banach spaces and AWS are further discussed in \cite{kuo, stroock}. Classical examples for which the AWS are the natural underlying probability spaces include (fractional) Brownian motion, (fractional) Gaussian fields, or space-time white noise, as detailed in \cite{finlayson, chiusole} and the references therein. The simplest case of Brownian motion is precisely the example introduced by Cameron and Martin.
AWS are still an active area of research. In fact, in the recent paper \cite{chiusole}, the authors enhance the AWS formalism to cover key results, such as large deviations, Cameron-Martin formulas, and Fernique's concentration of measure theorem, within the context of Gaussian rough paths and regularity structures.

From the mathematical physics perspective, there is an interesting link between abstract Wiener spaces and (free) Euclidean field theory, see e.g.~Section 4.3 in \cite{stroock}. Important are Hilbert spaces which are invariant under the Euclidean group. In fact, fields are formulated as probability measures on function spaces, particularly measures on Schwartz distributions over the Euclidean space.
The path integral formulation requires defining a probability measure on a space of fields, often modelled as a Gaussian measure with additional interactions. 
The construction of a free Euclidean field is directly linked to an AWS, where the underlying Hilbert space is the space of solutions to a covariance operator. 
The AWS formalism provides a rigorous approach to Wiener integration in field theory. This mathematical framework plays a crucial role in ensuring the well-defined nature of the Euclidean path integral (in the non-interacting case) and provides a probabilistic interpretation of field theory in terms of stochastic processes.

Complex Gaussian random elements are important objects in many areas beyond mathematics. For example, they can be used to model quantum harmonic oscillators or electromagnetic fields in physics \cite{aboujaoude}, signals in sonar or radar systems in engineering such as signal processing \cite{lapidoth}, or signals in machine learning reconstruction algorithms. In random matrix theory, Ginibre ensembles are matrices with i.i.d.~complex Gaussians entries, see e.g.~\cite{anderson}. 

Extending the framework of real abstract Wiener spaces ($\R$-AWS) to the complex setting when working with complex Gaussians can be done in two different ways. In \cite{taniguchi, kusuoka, shigekawa}, the authors respectively consider holomorphic functions, pseudo-convex domains and It\^o calculus on so-called {\emph almost} complex abstract Wiener spaces ($\mathbb{C}$-AWS). 

An almost $\mathbb{C}$-AWS is essentially an $\R$-AWS with an additional linear operator 
$J: H \rightarrow H$
such that $J^2=-\operatorname{id}$, mimicking multiplication with the complex number $i$ on $H$ without assuming that $H$ is intrinsically complex. In some sense $J$ is a complex structure added on top of the real space which does not need to be integrable. Almost $\mathbb{C}$-AWS are natural spaces for modelling intrinsic real Gaussian fields or measures which exhibit  complex-like structures. For example, one may consider integration in Euclidean quantum field theory for those fields which are intrinsically real. There Wick rotation introduces complex structures giving rise to such an almost complex linear operator $J$. 

In particular, in \cite{shigekawa} the author proves that the It\^o-Wiener expansion of an $L^p$-holomorphic function on a complex abstract Wiener space converges strongly in the $L^p$ norm. They establish that the space of $L^p$-holomorphic functions is the closure of holomorphic polynomials, ensuring their density. Additionally, they present a unicity theorem stating that if two such functions agree on a set of positive measure, they are equal almost everywhere. The work of \cite{kusuoka} extends the concept of pseudoconvexity to almost complex abstract Wiener spaces, providing a framework for infinite-dimensional holomorphic function theory. The authors in \cite{taniguchi}  generalize the notion of almost complex structures to abstract Wiener spaces, even when the complex structure does not extend continuously to the entire Banach space. They assume that $J$ is bounded and linear and extends continuously to the whole space. The paper constructs a new complex abstract Wiener space from a given almost complex structure and uses this to define holomorphic functions, their skeletons (restrictions to the Cameron–Martin space), and a version of the Cauchy-Riemann operator via Malliavin calculus.

In this work we directly define 
complex abstract Wiener spaces. Under the assumptions of \cite{taniguchi} ($J$ is linear bounded operator, it extends to a continuous linear operator on the whole Banach space and preserves the Gaussian measure)
 this gives an equivalent construction. Our advantage is that the direct approach gives an intrinsic global complex structure and needs not verify the assumptions and existence of the operator $J$.

While on flat spaces those approaches differ only conceptually, the almost complex Wiener spaces become relevant if the Banach space is replaced by a manifold, \cite{beltita}. In those cases $J$ may be non-linear and even non-integrable \cite{lempert}. 
 
In our case, the Hilbert and Banach spaces are assumed to be complex, and $\nu$ is a measure which is defined over the complex space $B$. The $\mathbb{C}$-AWS is a natural framework when working  with truly complex fields, which appear for example in the studies  of charged particles in quantum field theories or the Fock space representation of quantum free fields (also known as the complex Gaussian free field).\\

The aim of this paper is firstly to present a rigorous, complete and self-contained general framework for $\mathbb{K}$-AWS, where $\mathbb{K} \in \{\mathbb{R},\mathbb{C}\}$. In particular, we want to study the complex case $\mathbb{K}=\mathbb{C}$, and extend Fernique's theorem to hold for the complex case. While Gross' original article defines the AWS using measurable norms, we follow the definition by way of the characteristic function as given in \cite{stroock}. Even more than in \cite{stroock}, we focus on the perspective of Gaussian fields to highlight the connection with integral calculus. In this paper we study symmetric Gaussians, which in the real case are centred (exhibiting mirror symmetry around 0) and in the complex case are proper (exhibiting circular symmetry around 0). In fact, if $H$ is a separable $\K$-Hilbert space, $B$ is a separable $\K$-Banach space containing $H$, and $\nu$ is a probability measure on $B$ we show the following.

\begin{thm*}[Covariance as trace class operator]
$X$ is a symmetric $H$-valued Gaussian field over $\mathbb{K}$ iff the covariance function can be written in terms of some non-negative, self-adjoint trace class operator.
\end{thm*}

\begin{thm*}[Existence and uniqueness]
For any $H$ there exist $B$ and $\nu$ such that $(H, B, \nu)$ is a $\K$-AWS. Conversely, if $\nu$ is a symmetric Gaussian measure on $B$, then there exists a unique $H$ such that $(H, B, \nu)$ is a $\K$-AWS.
\end{thm*}

Essentially, this result says that Gaussian fields are in some sense equivalent to abstract Wiener spaces. For our third and final result we will relate the $\C$-AWS to the $\R$-AWS by way of a real structure, which is a real linear, complex anti-linear involution on a complex vector space. If there is a bounded real structure $\sigma$ on a complex Banach space $B$, then $B$ can be written as the vector space direct sum $B_\R \oplus i B_\R$, where $B_\R$ is the invariant set under $\sigma$ endowed with the $B$-norm, and furthermore the topology on $B$ is equivalent to the product topology on $B_\R \times B_\R$. Assuming additionally that $\sigma$ is a bounded real structure on $B$ which is isometric when restricted to $H$, we show that:

\begin{thm*}
$Z \sim \nu$ on the $\mathbb{C}$-AWS iff $X=\sqrt{2}\Re Z$ and $Y=\sqrt{2}\Im Z$ are i.i.d.~on some $\mathbb{R}$-AWS.
\end{thm*}

In fact our third result leads to the following commutative relation between the real and complex Gaussian fields and the real and complex abstract Wiener spaces.

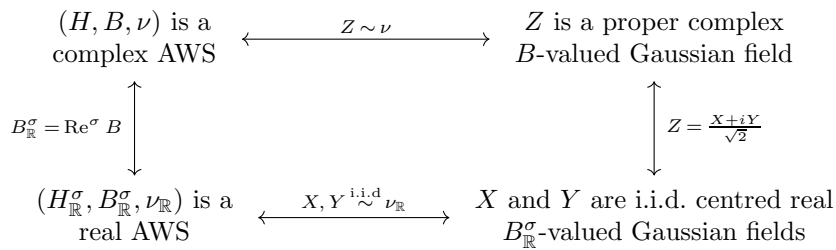
\begin{center}
    \begin{tikzcd}
        \begin{tabular}{c}
            $(H, B, \nu)$ is a \\
            complex AWS
        \end{tabular}
        \arrow[rrr, "{Z \,\sim\, \nu}", leftrightarrow]
        \arrow[dd, "{B_\R \,=\, \Re B\;}", leftrightarrow, ']
        &&&
        \begin{tabular}{c}
            $Z$ is a proper complex \\
            $B$-valued Gaussian field
        \end{tabular}
        \arrow[dd, "{\;Z \,=\, \frac{X + i Y}{\sqrt2}}", leftrightarrow]
        \\ \\
        \begin{tabular}{c}
            $(H_\R, B_\R, \nu_\R)$ is a \\
            real AWS
        \end{tabular}
        \arrow[rrr, "{X, \,Y \,\overset{\text{i.i.d}}{\sim}\, \nu_\R}", leftrightarrow]
        &&&
        \begin{tabular}{c}
            $X$ and $Y$ are i.i.d.~centred real\\
            $B_\R$-valued Gaussian fields
        \end{tabular}
    \end{tikzcd}
\end{center}

Furthermore we will construct concrete examples which fall into our framework of $\mathbb{C}$-AWS. First we use $\mathbb{C}$-AWS to construct  complex Brownian motion and we will show a complex version of the Feynman-Kac formula. Finally  we construct the complex fractional Gaussian fields. The complex Feynman-Kac formula and complex fractional Gaussian fields do not appear earlier in the literature to the best of the authors knowledge. It would be interesting to develop fundamental results in this setting like Schilder's theorem, large (moderate) deviation results, and differential Malliavin calculus in the future. 

The organisation of the paper is as follows. 
In Section \ref{sec:pre} we will define $\K$-random elements, $\K$-AWS, and real structures. Section \ref{sec:results} contains the precise results which we will prove in Section \ref{sec:proofs}. Finally  Section \ref{sec:examples} is devoted to our two examples and a complex version of the Feynman-Kac formula.

\section{Prerequisites}\label{sec:pre}
In this section we will define Banach space valued random elements. 
As the Banach space in question may be complex, we define complex random variables and discuss language to describe both real and complex random variables in Section \ref{ssec:cx_rv}. Within this general framework we define Gaussian fields over $\K \in \{\R, \C\}$ in Section \ref{ssec:gaussian_fields}, and $\K$-abstract Wiener spaces in Section \ref{ssec:AWS}. Finally, Section \ref{ssec:real_struc} introduces the necessary structure to compare complex Gaussian fields to real ones.

We will not explicitly talk about the underlying probability space $(\Omega, \F, \P)$ with integral $\E = \int_\Omega \cdot \d\P$. We will assume it to be large enough to support our random elements. 

With $\K$ we mean either $\R$ or $\C$, in the sense that the statement in question holds for both fields. For example, for $x, y \in \K^d$ with $d \geq 1$ the Euclidean inner product is denoted $x \cdot y = x \bar y^T$ and the Euclidean norm is denoted $|x| = \sqrt{x \cdot x}$. The dimension of $\K$ over $\R$ is denoted throughout by $\kappa$, so $\kappa = 1$ when $\K = \R$ and $\kappa = 2$ when $\K = \C$.

If $B$ is a Banach space, its norm is denoted $\norm \cdot_B$ unless specified otherwise. Furthermore, the topological dual to $B$ is denoted $B'$, with norm 
\[ \norm{f}_{B'} := \opnorm{f} := \sup_{x \in B \setminus \{0\}} \frac{|f(x)|}{\norm{x}_B} = \sup_{x \in B: \norm{x}_B = 1} |f(x)| \]
for $f \in B'$. This is the \emph{operator norm}, defined in general for linear operators from one Banach space to another.

\subsection{\texorpdfstring{$\K\kern-9.75pt{\K}\kern1pt$}{K}-random variables}\label{ssec:cx_rv}

In this section we briefly define complex (Gaussian) random variables, and refer for a more detailed discussion to \cite{lapidoth}. Then we will describe both real and complex random variables with the same language. This will enable us to be more general in the rest of this paper.

A \emph{complex random variable} is a Borel measurable map $Z$ from a probability space to $\C$, which by the identity
\[ \B(\C) \cong \B(\R \times \R) = \B(\R) \otimes \B(\R) \]
may be defined as the complex linear sum $Z = X + i Y$, where $X$ and $Y$ are real random variables. The \emph{characteristic function} of $Z$ is defined to be
 \[ \phi_Z: \C \to \C, \quad w \mapsto \E \exp \big[ i \Re(Z \cdot w) \big], \]
which implies $\phi_{Z}(w) = \phi_{(X, Y)}(x, y)$ for all $x + i y = w \in \C$. Hence the distributions of $Z = X + i Y$ and $(X, Y)$ are equivalent, so properties like independence, uniqueness, or Gaussianity are inherited directly from the theory of real random vectors.

For a complex random variable $Z$ there is a difference between central moments and absolute central moments. We define the \emph{variance} and \emph{pseudo-variance} of $Z$ to be
\[ \Var Z = \E |Z - \E Z|^2, \qquad \PVar Z = \E (Z - \E Z)^2. \]
The covariance and pseudo-covariance are obtained by applying the complex and real polarisation identities, respectively.

\begin{example}
    A complex random variable $Z = X + i Y$ is said to be \emph{Gaussian} when the real random variables $X$ and $Y$ are jointly Gaussian.\footnote{We include random variables with variance zero as degenerate Gaussians.} In this case the characteristic function of $Z$ is given by
    \[ \phi_Z(w) = \exp \Big[ i \Re(\mu \cdot w) - \tfrac14 \big( \sigma^2|w|^2 + \Re(\rho \cdot w^2) \big) \Big], \qquad w \in \C, \]
    where $\mu$, $\sigma^2$, and $\rho$ are the mean, variance, and pseudo-variance of $Z$, respectively. We write $Z \sim \mc{CN}(\mu, \sigma^2, \rho)$. The standard complex Gaussian distribution is $\mc{CN}(0, 1, 0)$.
\end{example}

If a centred (mean zero) and finite variance complex random variable $Z$ additionally satisfies $\PVar Z = 0$, it is said to be \emph{proper}. In the case of Gaussians, proper complex random variables are precisely those that can be obtained by a scaling of a standard complex Gaussian.

\begin{proposition}[Proposition 24.2.11 in \cite{lapidoth}]\label{prop:proper_gaussian}
    Let $W$ be a standard complex Gaussian random variable.
    \begin{enumerate}
        \item For every centred complex Gaussian $Z$ there exist $\alpha, \beta \in \C$ such that $Z \sim \alpha W + \beta \bar W$.
        \item A centred complex Gaussian $Z$ is proper if and only if there is some $\alpha \in \C$ such that $Z \sim \alpha W$.
    \end{enumerate}
\end{proposition}

From this result we observe that a proper complex Gaussian $Z$ has the added property of being \emph{circularly symmetric}, i.e., $Z \sim \alpha Z$ for any $\alpha \in \C$ with $|\alpha| = 1$. A centred real Gaussian $X$ has a similar symmetry, namely $X \sim -X$. For this reason we will say that a $\K$-valued Gaussian random variable $X$ is \emph{symmetric (around zero)} if it is centred when $\K = \R$ and proper when $\K = \C$.\footnote{Any circularly symmetric complex random variable with finite variance is proper, but the converse need not hold (see Section 24.2.3 in \cite{lapidoth}). However, in the case of complex Gaussians the two concepts are equivalent.}

Recall also that $\kappa$ is defined to be the dimension of $\K$ over $\R$. With this, we have the language to describe random variables in $\R$ and $\C$ simultaneously. For example, a $\K$-random variable is a Borel measurable map from a probability space to $\K$, with characteristic function
\[ \phi_X: \K \to \C, \quad a \mapsto \E \exp \big[ i \Re(X \cdot a) \big]. \]
In particular, the characteristic function of a symmetric $\K$-Gaussian is given by $a \mapsto \exp( -\frac1{2\kappa} \sigma^2 |a|^2)$, with $\sigma^2 = \Var X$.

\subsection{\texorpdfstring{$\K\kern-9.75pt{\K}\kern1pt$}{K}-Gaussian fields}\label{ssec:gaussian_fields}

We define (Gaussian) random elements that take values in a real or complex Banach space $B$. In this section and the next we assume the most general case, i.e., that $B$ is a separable Banach space over $\K \in \{\R, \C\}$, and we will talk about $\K$-random fields or $\K$-Gaussian fields.

A thorough discussion of Banach space valued measurable functions and the concept of Bochner integration can be found in \cite{hytonen}. Many results of classical measure and integration theory can also be obtained for Bochner integrals over separable Banach spaces. We will use these without proof when the proof is a one to one generalisation of the classical case.

\begin{definition}\label{def:random_field}
    A map $X: \Omega \to B$ is called a \emph{$B$-valued random field} on $\Omega$ if it is $\F / \B(B)$ measurable. Equivalently, $X$ is a $B$-valued random field if $f(X)$ is a $\K$-random variable for every $f \in B'$. We also say that $X$ is a $\K$-random field.
\end{definition}

The equivalence in the definition above is due to the separability of $B$, see e.g.\ Theorem 1.1.6 in \cite{hytonen}. We will rely heavily on this weaker form of measurability. 

The distribution of a $B$-valued random field $X$ is completely determined by its characteristic function, defined to be
\[ \phi_X: B' \to \C, \quad f \mapsto \E \exp \big[ i \Re{}\hspace{-1pt} \circ f(X) \big]. \]
Like the characteristic function, the first and second moments of $X$ are defined weakly. In particular, the \emph{mean}, \emph{covariance}, and \emph{pseudo-covariance} functions are given by
\begin{gather*}
    m_X(f) = \E f(X), \\
    c_X(f, g) = \cov \big( f(X), g(X) \big), \\
    p_X(f, g) = \pcov \big( f(X), g(X) \big)
\end{gather*}
for all $f, g \in B'$. These functions are linear, sesquilinear, and bilinear, respectively. 

Again due to the equivalence in \Cref{def:random_field} we can obtain properties like independence, uniqueness, and Gaussianity weakly.

\begin{example}\label{ex:gaussian_field}
    A $B$-valued random field $X$ is \emph{Gaussian} if $f(X)$ is a $\K$-Gaussian random variable for every $f \in B'$. In that case the characteristic function of $X$ satisfies
    \[ \phi_X(f) = \exp\Big[ i \Re{}\hspace{-1pt} \circ m_X(f) - \tfrac14 \big( c_X(f, f) + \Re{}\hspace{-1pt} \circ p_X(f, f) \big) \Big], \qquad f \in B'. \]
    In particular, a symmetric Gaussian field $X$ has characteristic function 
    \[
    \phi_X: B' \to \C, \quad f \mapsto \exp \big[ -\tfrac1{2\kappa} c_X(f, f) \big]. \qedhere
    \]
\end{example}

If $X$ and $Y$ are two $B$-valued random fields, then $(X, Y)$ is a $B \times B$-valued random field. For each $h \in (B \times B)'$ the functions $f := h( \cdot, 0)$ and $g := h(0, \cdot)$ are continuous linear maps $B' \to \K$. Hence $(B \times B)' = B' \times B'$ and the characteristic function of $(X, Y)$ is denoted
\[ \phi_{(X, Y)}(f, g) = \E \exp \big[ i \Re \big( f(X) + g(Y) \big) \big], \quad f, g \in B'. \]
If $X$ and $Y$ are Gaussian, they are independent if and only if $\phi_{(X, Y)}(f, g) = \phi_X(f)\phi_Y(g)$ for all $f, g \in B'$. In that case, $(X, Y)$ is a Gaussian field on $B \times B$.

Before we move on to the topic of abstract Wiener spaces, we include one important result on the integrability of centred Gaussian fields. The following theorem implies that a Gaussian random field in fact belongs to $L^p(B)$, for every $p \in [1, \infty)$. It was shown by Xavier Fernique in 1970 for real Gaussian fields. A proof of the real case can be found in \cite{stroock}. We will sketch this proof, in particular where edits need to be made to include the complex case.

\begin{theorem}[Fernique's Theorem]\label{thm:fernique}
    Let $B$ be a separable Banach space over $\K$ and let $X$ be a centred $B$-valued Gaussian field. Then there exist constants $\alpha, K \in (0, \infty)$ such that
    \[ \E \exp \left( \alpha \norm{X}_B^2 \right) \leq K. \]
\end{theorem}

\begin{proof}[Proof sketch]
    As in \cite{stroock}, let $X'$ be an independent and identically distributed copy of $X$ (possibly on a larger probability space), and define
    \[ Y := \frac{X + X'}{\sqrt 2} \qquad \text{ and } \qquad Y' := \frac{X - X'}{\sqrt 2}. \]
    Fix $f, g \in B'$ and write $\eta := (f+g) / \sqrt2$ and $\zeta := (f-g)/\sqrt 2$, then $\eta, \zeta \in B'$. Now $f(Y) + g(Y') = \eta(X) + \zeta(X')$, so by independence of $X$ and $X'$,
    \begin{align*}
        \phi_{(Y, Y')}(f, g) &= \E \exp \Big[ i \Re \big( f(Y) + g(Y') \big) \Big] = \E \exp \Big[ i \Re \big( \eta(X) + \zeta(X') \big) \Big] \\
        &= \phi_{(X, X')}(\eta, \zeta) = \phi_X(\eta) \phi_{X'}(\zeta).
    \end{align*}
    As $X\sim X'$, they have the same sesquilinear covariance function $c$ and bilinear pseudo-covariance function $p$. By a straightforward calculation
    \begin{align*}
        c(\eta, \eta) + c(\zeta, \zeta) &= c(f, f) + c(g, g), \\
        p(\eta, \eta) + p(\zeta, \zeta) &= p(f, f) + p(g, g),
    \end{align*}
    from which it follows that $\phi_{(Y, Y')} = \phi_{(X, X')}$, and thus $(Y, Y') \sim (X, X')$.  From here the rest of the proof is identical to that given in \cite{stroock}.
\end{proof}

As a result of Fernique's Theorem, we see that any centred Gaussian field is an element of $L^p(B)$, so in particular each centred Gaussian field is Bochner integrable. 

\begin{remark}
If $X$ is a Bochner integrable $B$-valued Gaussian field then there exists some $\mu \in B$ such that $f(\mu) = \E[f(X)]$ for every $f \in B'$ (see Theorem 3.2.3 in \cite{stroock}). In particular, $X - \mu$ is a centred Gaussian field.

It is not clear a priori how to make an improper Gaussian field proper. In view of \Cref{prop:proper_gaussian} a first requirement would be the existence of a real structure on $B$, see \Cref{ssec:real_struc}. Recall also that in the case of Gaussians, properness is equivalent to circular symmetry. It is not clear whether there exists an invertible operator mimicking a global rotation that makes all components circularly symmetric.
\end{remark}

\subsection{Abstract Wiener spaces}\label{ssec:AWS}

From here on out we will limit our discussion to symmetric $\K$-Gaussian fields. A prime example of such a field is the standard Gaussian field, defined over a Banach space $B$ by the characteristic function
\[ \phi(f) = \exp \left( - \frac1{2\kappa} \norm{f}_{B'}^2 \right), \qquad f \in B'. \]
We will see that a standard Gaussian field on a Hilbert space $H$ has almost surely infinite $H$-norm. As an alternative, we consider the abstract Wiener space $(H, B, \nu)$. Here the Hilbert space $H$ is embedded into a larger Banach space $B$, on which we are able to define a Gaussian measure $\nu$ that looks like it would have, had it been the standard Gaussian measure on $H$. 

\begin{remark}
    Consider a Hilbert space $H$ over $\K$. By the Riesz Representation Theorem, any $f \in H'$ may be identified with some $h_f \in H$ such that $f(x) = (x, h_f)_H$ for all $x \in H$, and $\norm{f}_{H'} = \norm{h_f}_H$. We call $h_f$ the Riesz representative of $f$ in $H$. In most literature $H$ is identified with its dual (if $\K = \R$) or antidual (if $\K = \C$). To keep our language consistent we choose not to do this, and always denote the dual space to $H$ by $H'$.
\end{remark}

In the following definition we will assume that $H$ and $B$ are Hilbert and Banach spaces over $\K$, respectively, such that $H \sub B$ algebraically, densely, and topologically.\footnote{Here we use the language from \cite{grubb}, specifically Chapter 12.4 and the discussion of the Lax-Milgram lemma. This lemma can be of use in light of \Cref{thm:characterisation} and the application in \cref{ssec:fgf}.} This means that $H$ is a linear subspace of $B$, that $H$ lies dense in $B$, and that the $H$-norm is stronger than the $B$-norm. In particular, this implies that any $f \in B'$ restricted to $H$ is a continuous, linear map $f|_H: H \to \K$.

\begin{definition}\label{def:AWS}
    Let $H$ be a separable Hilbert space over $\K$ and let $B$ be a separable Banach space over $\K$. Assume that $H \sub B$ algebraically, densely, and topologically. If $\nu$ is a probability measure on $(B, \B(B))$ with characteristic function
    \[ \phi_\nu: f \mapsto \exp \left( -\frac1{2\kappa} \norm{f}_{H'}^2 \right), \qquad f \in B', \]
    then we say that $(H, B, \nu)$ is an \emph{abstract Wiener space}. 
\end{definition}

Notice that if $X \sim \nu$ with $(H, B, \nu)$ an abstract Wiener space, then $X$ is a symmetric $B$-valued Gaussian field with covariance function $(f, g) \mapsto (f, g)_{H'} = (h_g, h_f)_H$. 

The following lemma shows that while $H$ completely determines the covariance structure (and hence all) of $X$, the space $H$ is in fact a $\nu$-null set. We will give a sketch of the proof.

\begin{lemma}
    Let $(H, B, \nu)$ be an abstract Wiener space with $H$ infinite-dimensional. Then $\nu(H) = 0$. 
\end{lemma}

\begin{proof}[Proof sketch]
    The proof is based on an argument given in Section 3.1 of \cite{stroock}. First, the $H$-norm is extended to $B$ by setting $\norm{x}_H = \infty$ for $x \in B \setminus H$. Then it is a lower semi-continuous function on $B$, so $H = \{x \in B: \norm{x}_H < \infty\} \in \B(B)$ (see Chapter 7 in \cite{kurdila}). 
    
    With an orthonormal basis $(f_n)_{n \in \N}$ of $H'$ and the random field $X: B \to B, x \mapsto x$ one can construct a sequence of independent standard Gaussian random variables $X_n = f_n(X)$. Here we note that in both the real and complex case one has $\E \exp(-|X_n|^2) < 1$. With this one can show that $\{x \in B: \sum_n |X_n(x)|^2 < \infty\}$ is a $\nu$-null set, and finally that it contains $H$. 
\end{proof}

Suppose now that $\nu$ is a standard Gaussian measure on an infinite-dimensional Hilbert space $H$. Then $(H, H, \nu)$ is an abstract Wiener space, so by the lemma $\nu(H) = 0$. However, we also have $\nu(H) = 1$ because $\nu$ is a probability measure on $H$, which is a contradiction. Hence if $H$ is an infinite-dimensional Hilbert space, there exists no $H$-valued standard Gaussian field.

\begin{remark}
    In general, abstract Wiener spaces are constructed by making use of so-called measurable norms. These were introduced by Gross in \cite{gross_meas}, and in \cite{gross} it was shown that they can be used to construct an abstract Wiener space. Though we will not give a formal definition here, we will explain the concept.
    
    Consider a real, infinite-dimensional Hilbert space $H$. A measurable norm on $H$ is weaker than the $H$-norm, and it satisfies some particular conditions concerning finite-dimensional projections. The idea is then that if $B$ is the Banach space completion of $H$ under the measurable norm, then there exists a countably additive Gaussian measure $\nu$ on $B$ making $(H, B, \nu)$ into an abstract Wiener space. Gross further showed that any norm defined by
    \[ \norm{x} := \sqrt{(Ax, x)_H}, \qquad x \in H \]
    is measurable if $A$ is an injective trace class operator on $H$. This is a very useful result, especially in light of the connection between trace class and Hilbert-Schmidt operators. Note that this theory has only been defined for real spaces.

    An in-depth discussion of measurable norms requires many more prerequisites, so we will leave it here. All our further results on abstract Wiener spaces will be in terms of Gaussian fields as in \cite{stroock}, though the proof referenced in Section \ref{ssec:wiener_measure} relies on measurable norms.
\end{remark}

\subsection{Real structures}\label{ssec:real_struc}

Remember that a complex random variable $Z$ can be written as $X + i Y$, where $X$ and $Y$ are real random variables. The same can be done for a complex Banach space valued random element, provided the Banach space has a so-called real structure.

\begin{definition}
    A \emph{real structure} on a complex vector space $B$ is an anti-linear involution $\sigma: B \to B$. The invariant set under $\sigma$ is called the \emph{real part} $B_\R^\sigma := \{x \in B: \sigma(x) = x\}$ of $B$. The maps
    \[ \Re_\sigma: B \to B_\R^\sigma, \quad x \mapsto \frac{x + \sigma(x)}2; \qquad \Im_\sigma: B \to B_\R^\sigma, \quad x \mapsto \frac{x - \sigma(x)}{2i} \]
    project $B$ onto $B^\sigma_\R$.
\end{definition}

Notice that multiple real structures can exist on the same complex vector space, see \Cref{ex:unbounded_real_structure}. When there is no confusion about which real structure is meant, we leave out the explicit denotion by $\sigma$.

\begin{remark}
    A complex vector space has a real structure if and only if it is the complexification of some real vector space, see Section 18.3 in \cite{gorodentsev}. In this case the real vector space is precisely the real part of the complex vector space. This means that a complex vector space $B$ endowed with a real structure can be written as the vector space direct sum
    \[ B = B_\R \oplus i B_\R \cong B_\R \times B_\R. \qedhere \]
\end{remark}

The above holds for complex vector spaces in general. From now on, we will restrict our focus to complex Banach spaces. If $B$ is a complex Banach space with a real structure, we equip $B_\R$ with the $B$-norm. Then $\br \times \br$ can be equipped with the product topology, so we define the product topology norm on $B$ by\footnote{This norm might be denoted $\norm \cdot_2$, but to avoid confusion with the $L^2(\Omega; B)$ norm we write $\ptnorm\cdot$.}
\[ \ptnorm{z} := \sqrt{\norm{\Re z}_B^2 + \norm{\Im z}_B^2}, \qquad z \in B. \]
We can now compare the norm topology on $B$ with the product topology on $\br \times \br$, i.e., we can see if the $B$-norm and the PT-norm are equivalent on $B$. Notice that if they are equivalent, then the real structure is bounded. As the following example illustrates, this is not always the case.

\begin{example}\label{ex:unbounded_real_structure}
    Let $B$ be an infinite-dimensional complex Banach space and let $\{x_j: j \in \mc J\}$ be a Hamel basis for $B$, where $\mc J$ is a necessarily uncountable index set. Without loss of generality assume that $\norm{x_j}_B = 1$ for all $j \in \mc J$. Consider a countable subset of $\mc J$ and identify it with $\N$, so $\N \sub \mc J$. Define now the map $\sigma: B \to B$ on the Hamel basis by
    \[ \sigma(x_j) = \begin{cases}
        \ell x_{2\ell} & \text{ if } j = 2\ell - 1 \in \N \\
        \frac1\ell x_{2\ell - 1} & \text{ if } j = 2\ell \in \N \\
        x_j & \text{ if } j \in \mc J \setminus \N,
    \end{cases} \]
    and extend it anti-linearly in the sense that
    \[ \sigma(x) = \sigma \left( \sum_{k = 1}^n c_{j_k} x_{j_k} \right) = \sum_{k=1}^n \oolsi{c_{j_k}} \sigma \left(x_{j_k}\right) \]
    for all $x \in B$. Then $\sigma$ is an unbounded real structure on $B$.

    Consider now the example $B = \ell^1(\C)$, the space of absolutely summable sequences in $\C$. This is an infinite-dimensional Banach space, so we can define the unbounded real structure $\sigma$ on $B$ as described above. However, there also exists a bounded real structure on $B$, namely coordinate-wise complex conjugation. Thus not only can multiple real structures exist on the same Banach space, but some of them may be bounded while at the same time others are unbounded. 
\end{example}

We give a couple of equivalent statements about the comparison of the topologies. More results on the comparison of norms can be found in \cite{moslehian}, though that is in the language of complexifications as opposed to real structures. The proof of the following proposition is elementary, so it is omitted.

\begin{proposition}
    Let $B$ be a complex Banach space with a real structure $\sigma$. Then the following are equivalent.
    \begin{enumerate}
        \item The map $\sigma:B \to B$ is bounded.
        \item The projection $\Re: B \to \br$ is continuous.
        \item The map 
        \[ \rho: B \to \br \times \br, \quad z \mapsto (\Re z, \Im z) \]
        is a topological vector space isomorphism.
        \item The $B$-norm and the PT-norm are equivalent on $B$.
    \end{enumerate}
\end{proposition}

Observe that a bounded real structure is \emph{isometric} with respect to the PT-norm. This is a useful property, especially in the consideration of Hilbert spaces. Indeed, if $\sigma$ is an isometric real structure on a complex Hilbert space $H$, then for all $x, y \in \hr$,
\[ (x, y)_H = \frac14 \sum_{k=0}^3 i^k \norm{x + i^k y}_H^2 = \frac14 \left( \norm{x + y}_H^2 - \norm{x-y}_H^2 \right). \]
This implies that $\hr$ is a real Hilbert space. 

One might think that given a complex Banach space $B$, the dual space $B'$ is immediately endowed with the real structure that is pointwise complex conjugation. However, this is not the case, since $\bar f$ for $f \in B'$ is complex anti-linear, so in general $\bar f \notin B'$. 

If $B$ is equipped with a bounded real structure $\sigma$, then one can show that $f \mapsto \olsi{f \circ \sigma}$ is a bounded real structure on $B'$. Furthermore, $\brp = (B_\R)' = (B')_\R$ can be thought of as the real part of $B'$ by way of complex linear extension.

\begin{remark}
    Suppose that $B$ is a complex Banach space with a bounded real structure, and fix $h \in B'$. Then $f:= \Re{}\hspace{-1pt} \circ h$ and $g:= \Im{}\hspace{-1pt} \circ h$ are continuous and real linear as maps $\br \to \R$, and we can write $h = f + i g$. However, the map
    \[ \br \times \br \to \R^2, \quad (x, y) \mapsto \big(f(x + i y), g(x + i y)\big) \]
    is continuous, but not necessarily real linear. 
\end{remark}

If $H$ is a complex Hilbert space with an isometric real structure $\sigma$, we have for any $f \in H'$ that $\Re h_f$ is equal to $h_{\Re f}$, the Riesz representative of $\Re f$ in $H$, and so
\[ \norm{\Re f}_{H'}^2 + \norm{\Im f}_{H'}^2 = \norm{\Re h_f}_H^2 + \norm{\Im h_f}_H^2 = \norm{h_f}_H^2 = \norm{f}_{H'}^2. \]
Thus the induced real structure $f \mapsto \overline{f \circ \sigma}$ on $H'$ is in fact also isometric.

Suppose now that $B$ is a complex Banach space with a bounded real structure and that $Z$ is a $B$-valued random field. Then the projections $\Re,$ $\Im: B \to \br$ are continuous and hence measurable, so $X:= \Re Z$ and $Y:= \Im Z$ are $\br$-valued random fields. For the converse, note that the map $\br \times \br \mapsto \C, (x, y) \mapsto (x + iy)$ is continuous and thus measurable, so if $X$ and $Y$ are random fields on a real Banach space $B_\R$, then $Z := X + i Y$ is a random field on the complex Banach space $B:= \br \oplus i \br$ endowed with the PT-norm.

From the fact that the real and imaginary parts of a complex Gaussian random variable are real jointly Gaussian random variables, we can conclude the following result for Gaussian fields.

\begin{lemma}
    Let $B$ be a complex Banach space with a bounded real structure. Suppose that $Z$ is a $B$-valued random field and write $X := \Re Z$ and $Y := \Im Z$.
    \begin{itemize}
        \item If $Z$ is a $B$-valued Gaussian field, then $X$ and $Y$ are $\br$-valued Gaussian fields.
        \item If $X$ and $Y$ are independent $\br$-valued Gaussian fields, then $Z$ is a $B$-valued Gaussian field.
    \end{itemize}
\end{lemma}

\section{Results}\label{sec:results}

In this section we will present our results. Theorems \ref{thm:characterisation} and \ref{thm:equivalence} are stated in general terms of $\K \in \{\R,\C\}$, whereas Theorem \ref{thm:independence} concerns the relation between real and complex abstract Wiener spaces.

Let $H$ be an infinite-dimensional, separable Hilbert space. The space of symmetric $H$-valued Gaussian fields can be completely characterised, because such fields are completely determined by their covariance functions. The following theorem is a more detailed version of Theorem 3.2.8 in \cite{stroock}, and it is extended to hold in general for $\K$-Gaussian fields.

\begin{theorem}\label{thm:characterisation}
    Let $H$ be an infinite-dimensional, separable Hilbert space over $\K$. 
    \begin{enumerate}[label = (\roman*)]
        \item \label{thm:characterisation:uniqueness}
        If $X$ is a symmetric $H$-valued Gaussian field, then its covariance function is given by $(f, g) \mapsto (Af, g)_{H'}$ for some non-negative, self-adjoint, trace class operator $A$ on $H'$. 
        \item \label{thm:characterisation:existence}
        If $A$ is a non-negative, self-adjoint, trace class operator on $H'$, then there exists a symmetric $H$-valued Gaussian field $X$ with covariance function $(f, g) \mapsto (Af, g)_{H'}$. 
    \end{enumerate}
    In either case, $X$ is non-degenerate if and only if $A$ is positive.
\end{theorem}

The next theorem describes an equivalence between symmetric Gaussian fields and $\K$-abstract Wiener spaces. It extends Theorem 3.3.2 in \cite{stroock} to hold in general for $\K$-vector spaces. Part 1 is restated to highlight the non-uniqueness of the Banach space, and the relation with \Cref{thm:characterisation}.

\begin{theorem}\label{thm:equivalence}
    \leavevmode
    \begin{enumerate}[label = (\roman*)]
        \item \label{thm:equivalence:operator}
        Let $H$ be a separable, infinite-dimensional Hilbert space over $\K$. For any positive, self-adjoint, Hilbert-Schmidt operator $A$ on $H'$, there exists a separable, infinite-dimensional Banach space $B$ over $\K$ and a probability measure $\nu$ on $(B, \B(B))$ such that $(H, B, \nu)$ is an abstract Wiener space.
        \item \label{thm:equivalence:measure}
        Let $B$ be a separable, infinite-dimensional Banach space over $\K$ with $\nu$ a probability measure on $(B, \B(B))$. Then $\nu$ is a symmetric Gaussian measure if and only if there is a separable, infinite-dimensional Hilbert space $H$ over $\K$ for which $(H, B, \nu)$ is an abstract Wiener space. In either case, $H$ is unique.
    \end{enumerate}
\end{theorem}

Finally, we have a result on the relation between real and complex abstract Wiener spaces, provided that the complex Banach and Hilbert spaces share an isometric real structure. This theorem is the most original result of our paper.

\begin{theorem}\label{thm:independence}
    Let $B$ be a complex Banach space and $H \sub B$ a complex Hilbert space. Assume that there exists a bounded real structure on $B$ such that its restriction to $H$ is an isometric real structure on $H$. Suppose that $Z$ is a $B$-valued random field and write 
    \[ X := \sqrt 2 \Re Z, \quad \text{ and } Y := \sqrt 2 \Im Z. \]
    Then the following are equivalent.
    \begin{enumerate}[label = (\roman*)]
        \item \label{thm:independence:complex} $Z$ is distributed by $\nu$ with $(H, B, \nu)$ a complex abstract Wiener space.
        \item \label{thm:independence:real} $X$ and $Y$ are independent and identically distributed by $\nu_\R$ with $(\hr, \br, \nu_\R)$ a real abstract Wiener space.
    \end{enumerate}
\end{theorem}

\section{Proofs}\label{sec:proofs}

\subsection{Proof of \texorpdfstring{\Cref{thm:characterisation}}{Theorem 3.1}}

For the proof of part \ref{thm:characterisation:existence} we will need the following lemma. This is a more detailed version of Corollary 3.2.7 in \cite{stroock} extended to hold in general for $\K$-Gaussian fields.

\begin{lemma}\label{lem:conserved_gaussianness}
    Let $B$ be a separable Banach space and let $(X_n)_{n \in \N}$ be a sequence of centred Gaussian fields. If $(X_n)_{n \in \N}$ converges in probability to some centred $B$-valued random field $X$, then also $X_n \to X$ in $L^p(B)$ for all $p \geq 1$. Furthermore, $X$ is Gaussian, and its (pseudo-){\allowbreak}covariance function is the uniform limit of the (pseudo-){\allowbreak}covariance function of $X_n$.
\end{lemma}

\begin{proof}
    Assume that $X_n \to X$ in probability. Then $X_n \to X$ in distribution also, which means that the sequence $(X_n)_{n \in \N}$ is tight. By the proof of Fernique's Theorem in \cite{stroock} there exist $\alpha, K \in (0, \infty)$ such that
    \[ \E \exp \left( \alpha \norm{X_n}_B^2 \right) < K \]
    for all $n \in \N$. Then for any $p \geq 1$ the collection $\{\lownorm{X_n - X}_B^p: n \in \N\}$ is uniformly integrable, and so $X_n \to X$ in $L^p$. 
    
    Next, for $f, g \in B'$ we have
    \begin{align*}
        \big|c_X(f, g) - c_{X_n}(f, g)\big| &= \Big| \E \Big[ f(X) \big( \overline{g(X) - g(X_n)} \big) + \big( f(X_n) - f(X) \big) \overline{g(X_n)} \Big] \Big| \\
        &\leq \norm{f}_{B'} \norm{g}_{B'} \norm{X - X_n}_2 \big( \norm X_2 + \norm{X_n}_2 \big)
    \end{align*}
    by the Cauchy-Schwarz inequality in $L^2(B)$. Since $X_n \to X$ in $L^2$ the sequence $(X_n)_{n \in \N}$ must be bounded in $L^2(B)$, so $c_{X_n}$ converges uniformly to $c_X$. The same argument can be made for the pseudo-covariance functions. 
    
    Finally, fix $f \in B'$, then $f(X_n) \to f(X)$ in distribution by the continuous mapping theorem. This implies a convergence of characteristic functions, so
    \begin{align*}
        \phi_X(f) &= \phi_{f(X)}(1) = \lim_{n \to \infty} \phi_{f(X_n)}(1) \\
        &= \lim_{n \to \infty} \exp \big( -\tfrac12 c_{X_n}(f, f) + \tfrac12 \Im p_{X_n}(f, f) \big) \\
        &= \exp \big( -\tfrac12 c_X(f, f) + \tfrac12 \Im p_X(f, f) \big).
    \end{align*}
    The distribution of a random field is completely determined by its characteristic function, so we conclude that $X$ is a centred $B$-valued Gaussian field.
\end{proof}

We can now prove \Cref{thm:characterisation}. This theorem is extended from Theorem 3.2.8 in \cite{stroock} to include the complex case. We follow the proof given there.

\begin{proof}[Proof of part \ref{thm:characterisation:uniqueness}]\let\qed\relax
    Let $X$ be a symmetric $H$-valued Gaussian field. Recall from \Cref{ex:gaussian_field} that its characteristic function $\phi_X$ is determined only by its covariance function $c_X$, in the form
    \[ \phi_X(f) = \exp \big( -\tfrac1{2\kappa} c_X(f, f) \big), \quad f\in H'. \]
    By Fernique's Theorem \ref{thm:fernique},
    \begin{align*} |c_X(f, g)| &\leq \E \Big[ \big|(X, h_f)_H \overline{(X, h_g)_H}\big| \Big] \leq \norm{h_f}_H \norm{h_f}_H \E \Big[ \norm{X}_H^2 \Big] \\
    &= \norm{f}_{H'} \norm{g}_{H'} \norm{X}_2^2 < \infty
    \end{align*}
    for all $f, g \in H'$, so $c_X$ is bounded. By definition, $c_X$ is a symmetric, non-negative definite sesquilinear form $H' \times H' \to \K$. There now exists a bounded, self-adjoint, non-negative linear operator $A: H' \to H'$ such that 
    \[ c_X(f, g) = (Af, g)_H \quad \text{ for all } \quad f, g \in H'. \]
    Left is to show that $A$ is trace class. As $H$ is an infinite-dimensional, separable Hilbert space, there exists an orthonormal basis $(x_n)_{n \in \N}$ of $H$. Denote by $(f_n)_{n \in \N}$ the associated basis of $H'$, so $x_n$ is the Riesz representative of $f_n$ for each $n \in \N$. We may now calculate by monotone convergence
    \begin{align*}
        \trace A &= \sum_{n=1}^\infty (Af_n, f_n)_{H'} = \sum_{n=1}^\infty c_X(f_n, f_n) = \sum_{n = 1}^\infty \E \Big[ \big| f_n(X) \big|^2 \Big] \\
        &= \E \Bigg[ \sum_{n = 1}^\infty \big| (X, x_n)_H \big|^2 \Bigg] = \E \norm{X}_H^2 = \norm{X}_2^2 < \infty,
    \end{align*}
    so $A$ is indeed trace class.
\end{proof}

\begin{proof}[Proof of part \ref{thm:characterisation:existence}]\let\qed\relax
    Let $A: H' \to H'$ be a non-negative, self-adjoint, trace class operator. In particular, $A$ is compact, so by the Hilbert-Schmidt Spectral Theorem there exists a sequence of eigenfunctions $(f_n)_{n \in \N}$ of $A$ that forms an orthonormal basis of $H'$. Furthermore, let $(x_n)_{n \in \N}$ be the associated basis of $H$, so $x_n$ is the Riesz representative of $f_n$ for every $n \in \N$. Finally, the corresponding sequence of eigenvalues $(\alpha_n)_{n \in \N}$ is real, and even non-negative by the non-negativity of $A$.
    
    Let $(X_n)_{n \in \N}$ be a sequence of independent standard $\K$-valued Gaussian random variables.\footnote{Symmetric Gaussians with variance one.} Define for each $m \in \N$ the $H$-valued random field
    \[ S_m := \sum_{n=1}^m \alpha_n^{\frac12} X_n x_n, \]
    By independence of the $X_n$,
    \begin{align*}
        \phi_{S_m}(f) = \E \exp \left( i \sum_{n=1}^m \alpha_n^{\frac12} X_n (f, f_n)_{H'} \right) = \exp \left( - \frac12 \sum_{n=1}^m \alpha_n \big|(f, f_n)_{H'}\big|^2 \right)
    \end{align*}
    for any $m \in \N$ and $f \in H'$. Furthermore, 
    \begin{align*} 
    c_{S_m}(f, g) &= \sum_{n=1}^m \sum_{k=1}^m \alpha_n^\frac12 \alpha_k^{\frac12} \E \Big[ X_n \overline{X_k} \Big] (f, f_n)_{H'} \overline{(g, f_k)_{H'}} \\
    &= \sum_{n=1}^m \alpha_n (f, f_n)_{H'} \overline{(g, f_n)_{H'}}
    \end{align*}
    for any $m \in \N$ and $f, g \in H'$. This implies that $(S_m)_{m \in \N}$ is a sequence of Gaussian fields. Notice now that for $k < m$, 
    \[ \norm{S_m - S_k}_2^2 = \E \sum_{n=k + 1}^m \alpha_n |X_n|^2 = \sum_{n=k + 1}^m \alpha_n \leq \sum_{n > k} \alpha_n. \]
    As $A$ is trace class, the sum $\sum_n \alpha_n$ converges and so the tail goes to zero. Thus $(S_m)_{m \in \N}$ is a Cauchy sequence in the Banach space $L^2(H)$, so there is a limit $S \in L^2(H)$. By \Cref{lem:conserved_gaussianness} we see that $S$ is a symmetric $H$-valued Gaussian field with covariance function
    \[ c_S(f, g) = \lim_{m \to \infty} c_{S_m}(f, g) = \sum_{n=1}^\infty \alpha_n (f, f_n)_{H'} \overline{(g, f_n)_{H'}} = (Af, g)_{H'}, \]
    for all $f, g \in H'$, i.e., a covariance function determined by $A$.
\end{proof}

The final statement of the theorem is an immediate result from the form of the covariance function in this set-up. \qed

\subsection{Proof of \texorpdfstring{\Cref{thm:equivalence}}{Theorem 3.2}}

\begin{proof}[Proof of part \ref{thm:equivalence:operator}]\let\qed\relax
    First we need to find some Banach space $B$ that contains $H$ algebraically, densely, and topologically. Then we need to show that there exists a proper $B$-valued Gaussian field with covariance function given by the $H'$-inner product. We will define $B$ by setting its dual to be the image of $A$ in $H'$. Then $B$ will turn out to be a Hilbert space, so we may apply Theorem \hyperref[thm:characterisation:existence]{3.1(ii)} to $A$ as an operator on $B'$.

    By its positivity, $A$ is a bijection onto $R(A)$, so $A^{-1}: R(A) \to H'$ is well defined. If we set $V = R(A)$ and 
    \[ (u, v)_V = \big( A^{-1}u, A^{-1} v \big)_{H'}, \qquad u, v \in V,  \]
    then $V$ is an inner product space. Suppose that $(u_n)_{n \in \N}$ is a Cauchy sequence in $V$, then $(A^{-1}u_n)_{n \in \N}$ is a Cauchy sequence in $H'$. Thus there exists some $f \in H'$ such that
    \[ \lim_{n \to \infty} \norm{A f - u_n}_V = \lim_{n \to \infty} \norm{f - A^{-1}u_n}_{H'} = 0, \]
    so $V$ is in fact a Hilbert space. Furthermore,
    \[ \norm{u}_{H'} = \norm{A A^{-1} u}_{H'} \leq \opnorm{A} \norm{A^{-1} u}_{H'} = \opnorm{A} \norm{u}_V \]
    for any $u \in V$, so $V \sub H'$ algebraically and topologically. Set $B := V'$, then $B$ is a Hilbert space and $H \sub B$ algebraically and topologically. From now we will write $B'$ as opposed to $V$.

    Suppose that $H$ is not dense in $B$. By the Hahn-Banach Theorem there must exist some $u \in B' \setminus \{0\}$ such that $u(h) = 0$ for all $h \in H$. However, then $(A^{-1}u)(h) = 0$ for all $h \in H$, which contradicts the fact that $u \neq 0$. Thus we may conclude that $H$ lies dense in $B$.

    Consider now $A$ as an operator on $B'$. For $u, v \in B'$ there exist $f, g \in H'$ such that $u = Af$ and $v = Ag$. Hence
    \[ (Au, v)_{B'} = (u, A^{-1} v)_{H'} = (Af, g)_{H'} = (f, Ag)_{H'} = (u, Av)_{H'}, \]
    from which it follows that $A$ is self-adjoint and positive on $B'$. 
    
    Let $(u_n)_{n \in \N}$ be an orthonormal basis for $B'$ and set $f_n := A^{-1} u_n$ for $n \in \N$, then $(f_n)_{n \in \N}$ is an orthonormal system in $H'$. For $f \in H$ we have
    \[ f = A^{-1} Af = A^{-1} \sum_{n \in \N} u_n (Af, u_n)_{B'} = \sum_{n \in \N} f_n(f, f_n)_{H'} \]
    by continuity and linearity of $A$ and $A^{-1}$. This means that $(f_n)_{n \in \N}$ is in fact an orthonormal basis for $H'$. Furthermore, $\norm{A u_n}_{B'} = \norm{A f_n}_{H'}$, so $A$ has the same Hilbert-Schmidt norm on $H'$ as it does on $B'$. This implies that $A$ is a Hilbert-Schmidt operator on $B'$ as well, and hence $A^2 = A^* A$ is a positive, self-adjoint trace class operator on $B'$. By Theorem \hyperref[thm:characterisation:existence]{3.1(ii)} there exists a symmetric $B$-valued Gaussian field with covariance function $(u, v) \mapsto (A^*A u, v)_{B'} = (Au, Av)_{B'} = (u, v)_{H'}$. Denoting the distribution of this Gaussian field by $\nu$, we see that $\nu$ is a probability measure on $(B, \B(B))$ with characteristic function 
    \[ \phi_\nu: B' \to \C, \quad u \mapsto \exp \left( - \frac12 \norm{u}_{H'}^2 \right). \]
    As $H \sub B$ algebraically, densely, and topologically, we conclude that $(H, B, \nu)$ is an abstract Wiener space.
\end{proof}

Part \ref{thm:equivalence:measure} can be proven by following the proof of Theorem 3.3.2 in \cite{stroock}, and making appropriate changes to account for complex values as is done in the proof of \Cref{thm:characterisation}. \qed

\subsection{Proof of \texorpdfstring{\Cref{thm:independence}}{Theorem 3.3}}

We make use of the following lemma, which is stated slightly more generally than strictly necessary for our purposes.

\begin{lemma}\label{lem:adt}
    Let $B$ be a complex Banach space and $H \sub B$ a complex Hilbert space. Assume that there exists a bounded real structure on $B$ such that its restriction to $H$ is an isometric real structure on $H$. Then $H \sub B$ algebraically, densely, and topologically if and only if $\hr \sub \br$ algebraically, densely, and topologically.
\end{lemma}

\begin{proof}
    Suppose first that $H \sub B$ algebraically, densely, and topologically. As the projection $\Re$ is real linear and $\hr = \br \cap H$, we see that $\hr$ is a real linear subspace of $\br$. For density, notice that for any sequence $(z_n)_{n \in \N} \sub H$ approximating some $x \in \br \sub B$ it follows by continuity of the real projection $\Re$ that $(\Re(z_n))_{n \in \N} \sub \hr$ converges in $\br$ to $x$. Finally, $\hr \sub \br$ topologically because $\hr$ and $\br$ inherit the $H$- and $B$-norms, respectively.
    
    For the converse, suppose that $\hr \sub \br$ algebraically, densely, and topologically. As vector spaces, $H$ and $B$ are the complexifications of $\hr$ and $\br$, respectively, so $H \sub B$ algebraically. For $z = x + i y \in B$ there exist sequences $(x_n)_{n \in \N}$ and $(y_n)_{n \in \N}$ in $\hr$ converging in $B$ to $x$ and $y$, respectively. By the triangle inequality $(x_n + i y_n)_{n \in \N} \sub H$ converges in $B$ to $z$. Finally, $H \sub B$ topologically by another application of the triangle inequality.
\end{proof}

We now prove \Cref{thm:independence}.

\begin{proof}[\ref{thm:independence:complex} $\implies$ \ref{thm:independence:real}:]\let\qed\relax
    Denote by $\nu$ the distribution of $Z$, and assume that $(H, B, \nu)$ is a complex abstract Wiener space. Take $f, g \in \brp$ and define $h: B \to \C$ by $h(z) = f(z) - i g(z)$, where $f$ and $g$ are extended complex linearly to $B$. Then
    \[ f(X) + g(Y) = \sqrt2 \Re\big( h(Z) \big) \sim \mc N \left( 0, \norm{h}_{H'}^2 \right), \]
    so $(X, Y)$ is a Gaussian field on $\br \times \br$. In particular we observe that $X$ and $Y$ are centred $\br$-valued Gaussian fields with characteristic function 
    \[ \phi_X(f) = \phi_Y(f) = \exp \left( -\frac12 \norm{f}_{H'}^2 \right), \quad f \in \brp. \]
    Again for $f, g \in \brp$ and $h:= f - ig$ we have
    \begin{align*}
        \phi_{(X, Y)}(f, g) &= \E \exp \big[ i\sqrt 2 \Re \big( h(Z) \big) \big] = \exp \left( -\frac14 \norm{\sqrt2 h}_{H'}^2 \right) \\
        &= \exp \left( -\frac12 \left[ \norm{f}_{H'}^2 + \norm{g}_{H'}^2 \right] \right) = \phi_X(f) \phi_Y(g),
    \end{align*}
    since the real structure on $H$ is isometric. As $X$ and $Y$ are Gaussian, this implies they are independent and identically distributed. If their law is denoted $\nu_\R$, it follows by \Cref{lem:adt} that $(\hr, \br, \nu_\R)$ is a real abstract Wiener space.
\end{proof}

\begin{proof}[\ref{thm:independence:real} $\implies$ \ref{thm:independence:complex}:]
    Assume now that $X$ and $Y$ are independent and identically distributed by $\nu_\R$, and that $(\hr, \br, \nu_\R)$ is a real abstract Wiener space. Fix $h \in B'$ and write $f := \Re(h)$ and $g:= -\Im(h)$.
    As $X$ and $Y$ are independent and identically distributed, and the real structure is isometric on $H$, we obtain
    \begin{align*}
        \phi_Z(h) &= \E \exp \left( i \frac1{\sqrt{2}} \big[ f(X) + g(Y) \big] \right) = \phi_X \left( \frac{f}{\sqrt 2} \right) \phi_Y \left( \frac{g}{\sqrt 2} \right) \\
        &= \exp \left( - \frac14\left[ \norm{f}_{H'}^2 + \norm{g}_{H'}^2 \right] \right) = \exp \left( - \frac14 \norm{h}_{H'}^2 \right).
    \end{align*}
    By \Cref{lem:adt} we may now conclude that $(H, B, \nu)$ is an abstract Wiener space.
\end{proof}

\section{Examples}\label{sec:examples}

In what follows we will give two examples of complex Gaussian fields constructed via abstract Wiener spaces. These examples are well-known, see for example Chapter I in \cite{kuo} and Chapter 6 in \cite{silvestri} in the real case. In our case the Gaussian fields in question are in fact complex. 

In this section we will discuss complex function spaces, i.e., spaces of functions $f: D \to \C$ with the domain $D$ some subset of $\R^d$ or $\C^d$, for some $d \geq 1$. Note that we will write $L^2(D)$ as short for $L^2(D,\C)$ contrary to earlier, when we used $L^2(B)$ as a short notation for $L^2(\Omega,B)$.
Other spaces of interest are the fractional Sobolev spaces $H^s(D)$ and the space of continuos functions $C(D)$. 

All these spaces have an intrinsic real structure, namely pointwise complex conjugation. They are all embedded in $\ms D'(D)$, the space of distributions, defined to be the topological dual to $\test$, the space of test functions. The action of a distribution $u \in \ms D'(D)$ on a test function $\phi \in \test$ is denoted by $u(\phi) = \brak{u, \phi}$. We assume knowledge of these spaces and their norms and/or topologies, including the Fourier transform. A recommended source is the first five chapters of \cite{grubb} for a thorough discussion.

\subsection{The complex Wiener measure}\label{ssec:wiener_measure}

One of the main examples of an abstract Wiener space is the so-called classical Wiener space. In this section we will construct a complex version, and show how it induces a complex Brownian motion. 

Fix $T > 0$. Let $H \sub H^1([0, T])$ be the subspace of Sobolev functions  $z: [0, T] \to \C$ with $z(0) = 0$. This is a Hilbert space under the inner product 
\[ (z, w)_H := (\partial z, \partial w)_{\lt}, \qquad z, w \in H. \]
It may be embedded algebraically, densely, and topologically into the Banach space $B$ of continuous functions $[0, T] \to \C$ starting in zero, equipped with the sup norm. Note that both $H$ and $B$ are separable.

There is an inherent isometric real structure on these function spaces, namely the (pointwise) complex conjugation. We see that
\[ B_\R = B \cap \R^{[0, T]} \qquad \text{ and } \qquad H_\R = H \cap \R^{[0, T]} = H \cap B_\R. \]
The following result is well known, see for example Theorem 5.3 in \cite{kuo} for a full proof. Note that the proof uses the theory of measurable norms, which is not discussed in this article.

\begin{proposition}
    There exists a probability measure $\nu_\R$ on $(\br, \B(\br))$ such that $(\hr,\allowbreak \br, \nu_\R)$ is an abstract Wiener space.
\end{proposition}

By \Cref{thm:independence} we know that there exists a measure $\nu$ on $(B, \B(B))$ such that $(H, B, \nu)$ is an abstract Wiener space. More specifically, take $X, Y \sim \nu_\R$ to be independent and define $Z := (X + i Y)/\sqrt 2$, then $Z$ has law $\nu$. Then for any $f, g \in B'$ the random variables $f(Z)$ and $g(Z)$ are proper Gaussians with covariance
\[ c_Z(f, g) = \E \big[ f(Z) \bar{g(Z)} \big] = (f, g)_{H'}. \]
Let $(\Omega, \F, \P)$ denote the underlying probability space of $Z$ and define the stochastic process $W = \{W_t: t \in [0, T]\}$ associated to $Z$ by setting $W_t(\omega) := Z(\omega)(t)$ for all $\omega \in \Omega$ and $t \in [0, T]$. 

For any $t \in [0, T]$ define a function $B \to \C$ by
\[ x \mapsto  x(t) = \int_{[0, T]} x \d \delta_t, \qquad x \in B, \]
where $\delta_t$ is the Dirac point measure in $t$. This function is continuous and linear, so in this way we may identify the set of Dirac point measures $\{\delta_t: t \in [0, T]\}$ with a subset of $B'$. By a slight abuse of notation, we denote the function above by $\delta_t$. We may now write $W_t = \delta_t(Z)$ for all $t \in [0, T]$.

Let us investigate how such a functional behaves on $H$. Define for $t \in [0, T]$ the function
\[ d_t: [0, T] \to \C, \quad s \mapsto \begin{cases}
    s & \text{ if } s \leq t \\
    t & \text{ if } s > t.
\end{cases} \]
Since $\partial d_t = \one_{[0, t]}$ we have $d_t \in H$, and furthermore,
\[ \delta_t(h) = h(t) = \int_0^t \partial h(s) \d s = \int_{[0, T]} \partial h \, \partial d_t = (h, d_t)_{H} \]
for all $h \in H$. By the Riesz Representation Theorem we obtain $\norm{\delta_t}_{H'} = \norm{d_t}_H = \sqrt t$.

We may now observe some properties of the process $W$. By definition of the space $B$ we know that $\P$-almost surely, $t \mapsto W_t$ is continuous and starts in zero. Next, for any $t \in [0, T]$ the random variable $W_t = \delta_t(Z)$ is a proper complex Gaussian with variance $\lownorm{\delta_t}_{H'}^2 = t^2$. Finally, fix $0 \leq s_1 < t_1 \leq s_2 < t_2 \leq T$. Then 
\begin{align*} 
\E \big[ (W_{t_2} - W_{s_2}) (\overline{W_{t_1} - W_{s_1}}) \big] &= \big( \delta_{t_2} - \delta_{s_2}, \delta_{t_1} - \delta_{s_1} \big)_{H'} \\
&= \big( d_{t_1} - d_{s_1}, d_{t_2} - d_{s_2} \big)_H = 0,
\end{align*}
so $W$ has independent increments. 
Summarising these properties, we see that $W$ is the complex Brownian motion. We have now shown the following result.

\begin{proposition}
    There exists a $B$-valued Gaussian field $Z$ such that the stochastic process $W_t := \int Z \d \delta_t$ is the complex Brownian motion.
\end{proposition}

Fix now $n \in \N$ and $0 =: t_0 < t_1 < \dots < t_n \leq T$, and define the complex random vectors
\[ X := \big( W_{t_1}, W_{t_2}, \dots, W_{t_n} \big) \quad \text{ and } \quad Y := \big( W_{t_1}, W_{t_2} - W_{t_1}, \dots, W_{t_n} - W_{t_{n-1}} \big). \]
Let $f: \C^n \to \R$ be a non-negative measurable function. Then
\begin{align*} 
\int_{B} f\big(z(t_1), \dots, z(t_n)\big) \d \nu(z) &= \E \big[ f(X) \big] \\
&= \E \big[ f(Y_1, Y_1 + Y_2, \dots, Y_1 + \dots + Y_n) \big]
\end{align*}
As $W$ is the complex Brownian motion, $Y$ is a vector of independent proper complex Gaussians with $\Var Y_k = t_k - t_{k-1}$ for $1 \leq k \leq n$. If $c$ is the relevant normalisation constant, we obtain by a change of variables
\begin{align*} \E \big[ f(X) \big] = c \int_{\C^n} &\,f(y_1, y_1 + y_2, \dots, y_1 + \dots + y_n) \\ &\,\exp \left( -\frac12 \sum_{k=1}^n \frac{|y_k|^2}{t_k - t_{k-1}} \right) \d y_n \cdots\, d y_1 \\ = c \int_{\C^n} &\,f(x_1, \dots, x_n) \exp \left( -\frac12 \sum_{k=1}^n \frac{|x_k - x_{k-1}|^2}{t_k - t_{k-1}} \right) \d x_n \cdots\, \d x_1, \end{align*}
where we take $x_0 := 0$.

Fix now some arbitrary $x_0 \in \C$ and denote by $C_{x_0}([0, T])$ the space of continuous functions $[0, T] \to \C$ with $x(0) = x_0$. This is not a vector space, but equipped with the sup norm it has a topology and hence a Borel $\sigma$-algebra. 

Defining the bijective shift $s_{x_0}: B \to C_{x_0}([0, T])$ by $z \mapsto z + x_0$, we see that $s_{x_0}(Z)$ is measurable with respect to the Borel $\sigma$-algebra of $C_{x_0}([0, T])$. We may now conclude the following.

\begin{proposition}
    Fix $x_0 \in \C$ and $T > 0$. Then there is a unique measure $\nu_{x_0}$ on the Borel $\sigma$-algebra of $C_{x_0}([0, T])$ such that for all $n \in \N$, $0 =: t_0 < t_1 < \dots < t_n \leq T$ and each non-negative measurable function $f: \C^n \to \R$ we have
    \begin{align*} \int_{C_{x_0}([0, T])} &\,f \big( x(t_1), \dots, x(t_n) \big) \d \nu_{x_0}(x) = \\ &= c\int_{\C^n} f(x) \exp \left( -\frac12 \sum_{k=1}^n \frac{|x_k - x_{k-1}|^2}{t_k - t_{k-1}} \right) \d x. \end{align*}
\end{proposition}
\begin{proof}
    Existence follows from the previous argument, where $\nu_{x_0} = \nu \circ s_{x_0}^{-1}$, and uniqueness follows by uniqueness of the abstract Wiener space $(H, B, \nu)$ and the fact that $W$ completely determines $Z$ and vice versa.
\end{proof}

This is a complex, one-dimensional version of Theorem 20.2 in \cite{hall}. The full version can be proven rigorously in a similar manner, though it requires a lot more notation. With that one can prove a complex analogue to an integrated form of the Feynman-Kac formula.

Let $V: \C \to \C$ be continuous and bounded, and define the Hamiltonian $\ms H := -\Delta + V$, this is a self-adjoint operator on the complex Hilbert space $L^2(\C)$. We consider the operator $\exp(-T \H)$ defined by the Trotter product formula (Theorem 20.1 in \cite{hall}) to be
\[ \exp \big( -T\H \big) = \lim_{n \to \infty} \left[ \exp \left( \frac{T\Delta}{n} \right) \exp \left( -\frac{TV}{n} \right) \right]^n. \]
Here $\exp(-TV/n)$ is a multiplication operator and
\[ \exp \left( \frac{T\Delta}{n} \right) f(x_0) = \sqrt{\frac{n}{\pi T}} \int_\C \exp \left( - \frac{n}{T}|x - x_0|^2 \right) \d x, \qquad x_0 \in \C, \]
for $f \in L^2(\C)$. Finally, note that the space $C([0, T])$ is Banach when endowed with the supnorm, as $[0, T]$ is compact. Hence we can consider probability measures defined on the Borel $\sigma$-algebra of $C([0, T])$. 

\begin{theorem}[Feynman-Kac formula]
    There exists a probability measure $\nu_T$ on $C([0, T])$ such that
    \[ \big( \exp(-T\H) f, g \big)_\lt = \int_{C([0, T])} \bar{g(x(0))} \exp \left( \int_0^T V\big(x(s)\big) \d s \right) f\big(x(T)\big) \d \nu_T(x) \]
    for all $f, g \in L^2(\C)$. 
\end{theorem}

\subsection{Complex fractional Gaussian fields}\label{ssec:fgf}

For this section, let $s > 0$ be a positive real number and $D \sub \rd$ be an open, bounded, non-empty set. The following proposition is a well-known result, the proof of which is an application of Weyl’s law, Theorem 3 in Section 6.3 and Theorem 1 in Section 6.5 of \cite{evans}.

\begin{proposition}\label{prop:weyl}
    There exists a nondecreasing sequence $(\lambda_n)_{n \in \N} \sub (0, \infty)$ and an orthonormal basis $(w_n)_{n \in \N} \sub H_0^1(D) \cap C^\infty(\bar{D})$ of $L^2(D; \R)$ such that each $w_n$ is a non-trivial weak solution to
    \[ -\Delta w_n = \lambda_n w_n \text{ in } D, \qquad w_n = 0 \text{ on } \partial D. \]
    Furthermore, the sequence $(n^2 \lambda_n^{-d})_{n \in \N}$ converges to some positive constant.
\end{proposition}

This set $(w_n)_{n \in \N}$ then also forms an orthonormal basis of $L^2(D)$. For $s > 0$ we define the fractional Laplacian of order $s$ to be the operator in $\ltd$ given by
\[ \fls u := \sum_{n \in \N} \lambda_n^s (u, w_n)_{\lt} w_n. \]
This definition is independent of the choice of orthonormal basis $(w_n)_{n \in \N}$. We denote the domain of $\fls$ by $\lso := \{u \in \ltd: \fls u \in \ltd\}$ and define
\[ (u, v)_\ls := \big( \fls u, \fls v \big)_\lt, \qquad u, v \in \lso. \]
This is an inner product making $\lso$ into a complex Hilbert space with orthonormal basis $v^s = (v^s_n)_{n \in \N}$ defined by $v^s_n := \lambda_n^{-s} w_n$. 

\begin{remark}\label{rem:ls_equivalence}
    The $\lso$ spaces concern fractionally differentiable functions, in a similar way to the fractional Sobolev spaces. In particular, one can show that if $D$ is Lipschitz, then $\lso \sub H_0^{2s}(D)$ with equality whenever $s \notin \{\frac14 k: k \in 2\N_0 + 1\}$, see Chapter 3 in \cite{mclean}. Furthermore, $\test \sub \ms L_t(D) \sub \lso$ for any $t \geq s > 0$.
\end{remark}

Let $(a_n)_{n \in \N} \sub \C$ be some sequence such that
\[ \sum_{n \in \N} \lambda_n^{-2s} |a_n|^2 < \infty. \]
Then $u := \sum_n a_n w_n \in \ls'(D) \sub \ms D'(D)$ and $\ds u := \sum_n \lambda_n^{-s} a_n w_n \in \ltd$. We may now define $\lsd$ to be the space of such distributions $u$. It is a Hilbert space when equipped with the inner product 
\[ (u, v)_\lsm := \big( \ds u, \ds v \big)_\lt, \qquad u, v \in \lsd, \]
and $v^{-s} = (v_n^{-s})_{n \in \N}$ defined by $v^{-s}_n := \lambda_n^s w_n$ is an orthonormal basis. Setting $\mc L_0(D) := \ltd$ we may write for $0 \leq s \leq t$, 
\[ \test \sub \ms L_t(D) \sub \lso \sub \ltd \sub \lsd \sub \ms L_{-t}(D) \sub \ms D'(D). \]
Note in particular that $\lsd \sub \ms L_{-t}(D)$ algebraically, densely, and topologically, which facilitates the construction an abstract Wiener space. To do this, we apply Theorem \hyperref[thm:equivalence:operator]{3.2(i)}.

\begin{proposition}
    Let $D \sub \rd$ be an open, bounded, non-empty set and fix $s > 0$. For any $t > s + \frac14 d$ there exists a probability measure $\nu_t$ on the complex space $\ms L_{-t}(D)$ such that $(\lsd, \ms L_{-t}(D), \nu_t)$ is an abstract Wiener space.
\end{proposition}

\begin{proof}
    Fix $\eps > 0$. Define the linear operator $A: \lsp(D) \to \lsp(D)$ by $f \mapsto f \circ (-\Delta)^{-\eps}$. Notice that $(-\Delta)^{-\eps}: \ms L_{-s-\eps}(D) \to \lsd$ is a Hilbert space isomorphism, so $A$ is well-defined and in fact a Hilbert space isomorphism onto its image $\ms L_{-s-\eps}'(D)$. For $f \in \lsp(D)$ and $u \in \ms L_{-s}(D)$,
    \[ \big( Af \big)(u) = f \big( (-\Delta)^{-\eps} u \big) = \big( (-\Delta)^{-\eps} u, \bar a_f \big)_{\lsm} = \big( u, (-\Delta)^{-\eps} \bar a_f \big)_{\ms L_{-s}}, \]
    where $a_f$ is the Riesz representative of $f$ in $\lsd$. Observe that $(-\Delta)^\eps a_f$ is the Riesz representative of $Af$ in $\ms L_{-s}(D)$. Furthermore, as $(\lambda_n)_{n\in\N}$ is a sequence of positive real numbers, $(-\Delta)^{-\eps}$ is a positive, self-adjoint linear operator on $\ms L_{-s}(D)$. This means that $A$ is a positive, self-adjoint linear operator on $\lsp(D)$. For $n \in \N$ set $f_n:= (\cdot, v_n^{-s})_{\lsm}$ where $(v_n^{-s})_{n \in \N}$ is the orthonormal basis of $\lsd$ mentioned above, then $(f_n)_{n \in \N}$ is the associated orthonormal basis of $\lsp(D)$. Hence we can write the Hilbert-Schmidt norm of $A$ as
    \[ \norm{A}_{\text{HS}}^2 = \sum_{n \in \N} \norm{Af_n}_{\lsp}^2 = \sum_{n \in \N} \norm{(-\Delta)^{-\eps} v^{-s}_n}^2_{\lsm} = \sum_{n \in \N} \lambda_n^{-2\eps}, \]
    which due to Weyl's law converges if and only if $\eps > \frac14 d$.

    Let us now write $t = s + \eps$, and assume that $t > s + \frac14 d$. Then the map
    \[ A: \lsp(D) \to \lsp(D), \quad f \mapsto f \circ (-\Delta)^{s-t} \]
    is a positive, self-adjoint Hilbert Schmidt operator with range $\ltp(D)$. By Theorem \hyperref[thm:equivalence:operator]{3.2(i)} and its proof we may now conclude that there exists a probability measure $\nu_t$ on the Borel $\sigma$-algebra of $\ms L_{-t}(D)$ such that $(\ms L_{-s}(D), \ms L_{-t}(D), \nu_t)$ is an abstract Wiener space.
\end{proof}

For $u = \sum_n a_n w_n \in \ms L_{-t}(D)$ and $\phi \in \test$ we may write
\[ \brak{u, \phi} = \sum_{n \in \N} a_n \brak{w_n, \phi} = \Big( (-\Delta)^{-t}u, (-\Delta)^{t} \bar \phi \Big)_{\lt} = \Big( u, (-\Delta)^{2t} \bar \phi \Big)_{\ms L_{-t}}. \]
For fixed $\phi$ this describes a continuous linear function of $u$, which we denote here by $g_{\phi} \in \ms L_{-t}'(D)$. Denote the restriction of $g_\phi$ to $\lsd$ by $\tilde g_\phi$, then for $u \in \lsd$,
\begin{align*} 
\tilde g_\phi(u) &= \Big( (-\Delta)^{-t} u, (-\Delta)^{t} \bar \phi \Big)_{\lt} \\
&= \Big( (-\Delta)^{-s} u, (-\Delta)^{s} \bar \phi \Big)_\lt = \Big( u, (-\Delta)^{2s} \bar \phi \Big)_{\lsm}
\end{align*}
so $\lownorm{\tilde g_\phi}_{\ms L_{-s}'} = \lownorm{ (-\Delta)^{2s} \phi}_{\lsm} = \norm{\phi}_{\ls}$. Suppose now that $Z_t$ is an $\ms L_{-t}(D)$-valued random field with law $\nu_t$, then 
\[ \E \big[ \brak{Z_t, \phi} \overline{\brak{Z_t, \psi}} \big] = \E \big[ g_\phi(Z_t) \bar{g_\psi(Z_t)} \big] = (\tilde g_\phi, \tilde g_\psi)_{\lsp} = (\phi, \psi)_{\ls}. \]
This covariance structure is precisely the covariance of a fractional Gaussian field. Note that $Z_t$ is a proper Gaussian, so it is completely determined by its covariance function. As we can see, the covariance function restricted to $\test$ no longer relies on $t$. Hence we can define the fractional Gaussian fields in the following way.

\begin{theorem}
    Take $s > 0$ and let $D \sub \rd$ be an open, bounded, non-empty set. There exists a complex Gaussian field $Z$ such that 
    \[ \brak{Z, \phi} \sim \mc{CN} \left(0, \norm{\phi}_{\ls}^2, 0 \right) \]
    for all $\phi \in \test$. This field $Z$ is called the \emph{complex fractional Gaussian field of order $s$}.
\end{theorem}

By \Cref{rem:ls_equivalence} we obtain for $s = \frac12$ the definition of a homogeneous Gaussian free field as stated in \cite{ruszel}.

\vspace{1cm}

\subsection*{Funding}
T.J.v.L. and W.M.R. are funded by the Vidi grant VI.Vidi.213.112 from the Dutch Research Council. 

\subsection*{Acknowledgements} 
T.J.v.L. thanks Jiedong Jiang for help with \Cref{ex:unbounded_real_structure}, and Joppe Stokvis for our many discussions on all aspects of the article. 
W.M.R. would also like to thank Vittoria Silvestri for useful discussions about abstract Wiener spaces.
Both authors thank the anonymous reviewers for their valuable input and clarifications.

\bibliographystyle{alpha}
\bibliography{arxiv_MPAG_bib}

@book{aboujaoude,
    author = {Abdo {Abou Jaoud\'e}},
    title = {The Paradigm of Complex Probability and Quantum Mechanics: The Quantum Harmonic Oscillator with Gaussian Initial Condition – The Momentum Wavefunction and the Wavefunction Entropies},
    booktitle = {Simulation Modeling},
    publisher = {IntechOpen},
    address = {Rijeka},
    year = {2023},
    chapter = {2},
    DOI = {https://doi.org/10.5772/intechopen.1001985}
}

@book{anderson, 
    address={Cambridge}, 
    series={Cambridge Studies in Advanced Mathematics}, 
    title={An Introduction to Random Matrices}, 
    publisher={Cambridge University Press}, 
    author={Anderson, Greg W. and Guionnet, Alice and Zeitouni, Ofer}, 
    year={2009}, 
    collection={Cambridge Studies in Advanced Mathematics},
    DOI = {https://doi.org/10.1017/CBO9780511801334}
}

@book{evans,
    author = {Evans, Lawrence C.},
    title = {Partial Differential Equations},
    series = {Graduate Studies in Mathematics},
    volume = {19},
    year = {1998},
    DOI = { https://doi.org/10.1090/gsm/019/07},
    address={Berkeley},
    publisher = {American Mathematical Society},
    edition = {first}
}

@book{gorodentsev,
    author = {Gorodentsev, Alexey L. },
    title = {Algebra I},
    publisher = {Springer},
    address = {Cham},
    year = {2016},
    DOI = {https://doi.org/10.1007/978-3-319-45285-2}
}

@book{grubb,
    author = {Grubb, Gerd},
    year = {2009},
    title = {Distributions and Operators},
    publisher = {Springer-Verlag},
    address = {New York},
    DOI = {https://doi.org/10.1007/978-0-387-84895-2},
    series = {Graduate Texts in Mathematics},
    volume = {252}
}

@book{hall,
    title = {Quantum Theory for Mathematicians},
    author = {Hall, Brian C.},
    series = {Graduate Texts in Mathematics},
    volume = {136},
    publisher = {Springer},
    address = {New York},
    year = {2013},
    DOI = {https://doi.org/10.1007/978-1-4614-7116-5}
}

@book{hytonen,
    title = {Analysis in Banach Spaces},
    author = {Tuomas {Hyt\"onen} and Jan van Neerven and Mark Veraar and Lutz Weis},
    year = {2016},
    publisher = {Springer},
    address = {Cham},
    volume = {I: Martingales and Littlewood-Paley Theory},
    DOI = {https://doi.org/10.1007/978-3-319-48520-1}
}

@book{kuo,
    author = {Hui-Hsiung Kuo},
    title = {Gaussian Measures in Banach Spaces},
    year = {1975},
    series = {Lecture Notes in Mathematics},
    voume = {463},
    publisher = {Springer},
    address = {Berlin, Heidelberg},
    edition = {1},
    DOI = {https://doi.org/10.1007/BFb0082007}
}

@book{kurdila,
    author = {Kurdila, Andrew J. and Zabarankin, Michael},
    title = {Convex Functional Analysis},
    year = {2005},
    series = {Systems \& Control: Foundations \& Applications},
    publisher = {Birkh\"auser},
    address = {Basel},
    DOI = {https://doi.org/10.1007/3-7643-7357-1}
}

@book{lapidoth,
    address={Cambridge}, 
    edition={2}, 
    title={A Foundation in Digital Communication}, 
    publisher={Cambridge University Press},
    author={Lapidoth, Amos}, 
    year={2017},
    DOI = {https://doi.org/10.1017/9781316822708}
}

@book{mclean,
    author = {McLean, William},
    title = {Strongly Elliptic Systems and Boundary Integral Equations},
    year = {2000},
    publisher = {Cambridge University Press},
    address = {Cambridge}
}

@incollection{shigekawa,
  author    = {Ichiro Shigekawa},
  title     = {It{\^o}--Wiener expansions of holomorphic functions on the complex Wiener space},
  editor    = {Mayer-Wolf, Eddy and Merzbach, Ely and Shwartz, Adam},
  booktitle = {Stochastic Analysis},
  publisher = {Academic Press},
  address   = {San Diego},
  year      = {1991},
  pages     = {459--473},
  isbn      = {978-0-12-481005-1},
  doi       = {10.1016/B978-0-12-481005-1.50030-2},
}

@book{stroock,
    author = {Stroock, Daniel W.},
    title = {Gaussian Measures in Finite and Infinite dimensions},
    publisher = {Springer},
    address = {Cham},
    DOI = {https://doi.org/10.1007/978-3-031-23122-3},
    series = {Universitext},
    year = {2023},
    ISBN = {978-3-031-23121-6}
}

@article{beltita,
    title={Integrability of almost complex structures on Banach manifolds},
    author={Daniel Beltita},
    year={2005},
    volume = {28},
    pages ={59--73},
    journal = {Annals of Global Analysis and Geometry},
    DOI={https://doi.org/10.1007/s10455-005-2960-z}
}

@article{CM1,
    author = {Cameron, R. H. and Martin, W. T.},
    title = {{Transformations of Wiener integrals under translations}},
    journal = {Annals of Mathematics},
    volume = {45},
    year = {1944},
    pages = {386--396},
    DOI = {https://doi.org/10.2307/1969276}
}

@article{CM2,
    author = {Cameron, R. H. and Martin, W. T.},
    title = {{Transformations of Wiener integrals under a general class of linear transformations}},
    journal = {Transactions of the American Mathematical Society},
    volume = {58},
    year = {1945},
    pages = {184--219},
    DOI = {https://doi.org/10.1090/s0002-9947-1945-0013240-1}
}

@article{chiusole,
    author = {Chiusole, G. and Friz, P.K.},
    title = {{Towards Abstract Wiener Model Spaces}},
    journal = {arXiv:},
    volume = {2401.00169},
    year = {2024},
}

@article{finlayson,
    author = {Finlayson, Henry C.},
    title = {{Two Classical Examples of Gross' Abstract Wiener Measure}},
    journal = {Proceedings of the American Mathematical Society},
    volume = {53},
    number = {2},
    year = {1975},
    pages = {337--340},
    DOI = {https://doi.org/10.2307/2040009}
}

@article{gross,
    author = {Gross, Leonard},
    title = {{Abstract Wiener Spaces}},
    year = {1967},
    journal = {Berkeley Symposium on Mathematical Statistics and Probability},
    pages = {31--42}
}

@article{gross_meas,
    author = {Gross, Leonard},
    title = {{Measurable Functions on Hilbert Space}},
    year = {1962},
    journal = {Transactions of the American Mathematical Society},
    volume = {105},
    pages = {372--390},
    DOI = {https://doi.org/10.1090/S0002-9947-1962-0147606-6}
}

@article{kusuoka,
    author = {S. Kusuoka and S. Taniguchi},
    issn = {0022-1236},
    journal = {Journal of Functional Analysis},
    number = {1},
    pages = {62-117},
    title = {{Pseudoconvex Domains in Almost Complex Abstract Wiener Spaces}},
    volume = {117},
    year = {1993},
    DOI = {https://doi.org/10.1006/jfan.1993.1123}
}

@article{lempert,
    author = {Lempert, L.},
    journal = {Journal of Differential Geometry},
    pages = {519--543},
    title = {{Loop spaces as complex manifolds}},
    volume = {38},
    year = {1993},
DOI ={https://doi.org/10.4310/jdg/1214454481}
}

@article{moslehian,
    author = {{Moslehian, M.S.} and {Mu\~noz-Fern\'andez, G.A.} and {Peralta, A.M.} and {Seoane-Sep\'ulveda, J.B.}},
    title = {{Similarities and differences between real and complex Banach spaces: an overview and recent developments}},
    journal = {Revista de la Real Academia de Ciencias Exactas, F\'isicas y Naturales. Serie A. Matem\'aticas},
    volume = {116},
    number = {88},
    year = {2022},
    DOI = {https://doi.org/10.1007/s13398-022-01222-8}
}

@article{ruszel,
    author = {Chiarini, Leandro and Ruszel, Wioletta M.},
    title = {{Stochastic homogenization of Gaussian fields on random media}},
    journal = {Annales Henri Poincar\'e},
    year = {2024},
    volume = {25},
    pages = {1869--1895},
    DOI = {https://doi.org/10.1007/s00023-023-01347-5}
}

@article{segal1,
    author = {Segal, I. E.},
    title = {{Tensor algebras over Hilbert spaces I}},
    journal = {Transactions of the American Mathematical Society},
    year = {1956},
    volume = {81},
    pages = {106--134},
    DOI = {https://doi.org/10.2307/1992855}
}

@article{segal2,
    author = {Segal, I. E.},
    title = {{Distributions in Hilbert space and canonical systems of operators}},
    journal = {Transactions of the American Mathematical Society},
    year = {1958},
    volume = {88},
    pages = {12--41},
    DOI = {https://doi.org/10.2307/1993234}
}

@article{silvestri,
    title = {{Fluctuation results for Hastings–Levitov planar growth}},
    author = {Silvestri, Vittoria},
    journal = {Probability Theory and Related Fields},
    year  = {2017},
    volume = {167},
    pages = {417--460},
    DOI = {https://doi.org/10.1007/s00440-015-0688-7}
}

@article{taniguchi,
    author = {Taniguchi, S.},
    journal = {Osaka Journal of Mathematics},
    pages = {189--206},
    title = {{On almost complex structures on abstract Wiener spaces}},
    volume = {33},
    year = {1996},
    DOI = {https://doi.org/10.18910/4608}
}

\end{document}